\documentclass[a4paper]{scrartcl}

\usepackage{lmodern}

\usepackage[T1]{fontenc}
\usepackage[utf8]{inputenc}

\usepackage{amssymb}
\usepackage{amsmath}
\usepackage{amsthm}

\theoremstyle{plain}
\newtheorem{theorem}{Theorem}[section]
\newtheorem*{theorem*}{Theorem}
\newtheorem{proposition}[theorem]{Proposition}
\newtheorem{lemma}[theorem]{Lemma}
\newtheorem{corollary}[theorem]{Corollary}

\theoremstyle{definition}
\newtheorem{definition}[theorem]{Definition}

\theoremstyle{remark}
\newtheorem{remark}[theorem]{Remark}

\newcommand{\chara}{\mathrm{char}}

\newcommand{\Cen}{\mathrm{Cen}}
\newcommand{\MR}{\mathrm{MR}}
\newcommand{\Mdeg}{\mathrm{MD}}
\newcommand{\fix}{\mathrm{fix}}

\renewcommand{\phi}{\varphi}

\begin{document}

\title{On the geometry of sharply 2-transitive groups}
\author{Tim Clausen \quad and \quad Katrin Tent\footnote{Both authors partially supported by the Deutsche Forschungsgemeinschaft (DFG, German Research Foundation) under Germany’s Excellence Strategy EXC 2044–390685587, Mathematics M\"unster: Dynamics-Geometry-Structure. }}

\maketitle

\begin{abstract}
We show that the geometry associated to certain non-split sharply 2-transitive groups does not contain a proper projective plane. For a sharply 2-transitive group of finite Morley rank we improve known rank inequalities for this geometry and conclude that a sharply 2-transitive group of Morley rank 6  must be of the form $K\rtimes K^*$ for some algebraically closed field $K$.
\end{abstract}

\section{Introduction}\label{sec:introduction}
  The standard example of a sharply 2-transitive group is of the form $\mathrm{AGL}_1(K)\cong K\rtimes K^*$ where $K$ is a field or more generally a near-field. A sharply 2-transitive group is called split if it is of this form, i.e. if it has a nontrivial abelian normal subgroup. By results of Zassenhaus and Jordan every finite sharply 2-transitive group is in fact split.   Recently the first examples of non-split infinite sharply 2-transitive groups were constructed by Rips, Segev, and Tent  in characteristic two \cite{rips-segev-tent}, and by Rips and Tent in characteristic zero  \cite{rips-tent}. However, these groups are not of finite Morley rank (see Section~\ref{sec:finiteMR} below). It is not known if non-split sharply 2-transitive groups of finite Morley rank exist. By the Algebraicity Conjecture by Cherlin and Zil'ber, an infinite simple group of finite Morley rank should be an algebraic group over an algebraically closed field. By \cite{epstein-nesin} this would imply that Frobenius groups of finite Morley rank split. Therefore sharply 2-transitive groups of finite Morley rank can be considered a test case for this conjecture. 

Recent results by Alt\i nel, Berkman, and Wagner \cite{altinel-berkman-wagner} show that any infinite sharply 2-transitive group of finite Morley rank and characteristic $2$ is split and that any infinite split sharply 2-transitive group of finite Morley rank of characteristic different from $2$ is of the form $\mathrm{AGL}_1(K)$ for some algebraically closed field $K$.

If $G$ is an infinite non-split sharply 2-transitive group of finite Morley rank such that $\chara(G) \neq 2$, then $G$ admits a point-line geometry on the set of its involutions which has been studied by Borovik and Nesin in Section 11.4. of \cite{borovik-nesin}. We show that $G$ must be simple if the lines in this geometry are strongly minimal. Moreover, we use this geometry to prove new rank inequalities for $G$. In particular, if $\MR(G) = 6$ then $G$ is split and hence of the form $\mathrm{AGL}_1(K)$ for an algebraically closed field $K$ of Morley rank 3.

Our geometric arguments are similar to those used by Fr\'econ in \cite{frecon} to show that there is no bad group of Morley rank 3.  

\section{Preliminaries}\label{sec:prelim}
A permutation group $G$ acting on a set $X$, $|X| \geq 2$, is called \emph{sharply 2-transitive} if it acts regularly on pairs of distinct points, or, equivalently, if $G$ acts transitively on $X$ and for each $x\in X$ the point stabilizer $G_x$ acts regularly on $X \setminus \{x\}$. For two distinct elements $x,y \in X$ the unique $g \in G$ such that $(x,y)^g = (y,x)$ is an involution. Hence the set $J$ of involutions in $G$ is non-empty and forms a conjugacy class.
We put $J^2=\{ij\colon i, j\in J\}$ and call the elements of $J^2$ \emph{translations} extending the terminology used in the standard examples of sharply 2-transitive groups.

The (permutation) characteristic of a group $G$ acting  sharply 2-transitively on a set $X$ is defined as follows: put $\chara(G) = 2$ if and only if involutions have no fixed points. If involutions have a (necessarily unique) fixed point,  the $G$-equivariant bijection $i \mapsto \fix(i)$ allows us to identify the given action of  $G$ on $X$ with the conjugation action of $G$ on $ J$. Thus in this case, the nontrivial translations also form a single conjugacy class.  We put $\chara(G)=p$ (or $0$) if  translations have order $p$ (or infinite order, respectively).
For standard examples this definition of characteristic agrees with the characteristic of the field.

The following are well-known properties of sharply 2-transitive groups:
\begin{remark}\label{lem:basic} Let $G$ be a sharply 2-transitive group of characteristic $\chara(G) \neq 2$.
\begin{itemize}
  \item [(a)] $\Cen(i)$ acts regularly on $J\setminus \{i\}$,
  \item [(b)] the set $J$ acts regularly on $J$, i.e. for any two involutions $i,j\in J$ there is a unique involution $k\in J$ such that $i^k=j$, and
  \item [(c)] $J^2 \cap \Cen(i) = \{1\}$ for all $i \in J$.
\end{itemize}
In particular, a nontrivial translation does not have a fixed point.
\end{remark}

The crucial criterion for the splitting of a sharply 2-transitive group is the following~\cite{neumann}:

\begin{theorem}\label{thm:neumann}
  A sharply 2-transitive group $G$ splits if and only if the set of translations $J^2$ is a subgroup of $G$ (and in that case, $J^2$ must in fact be abelian).
\end{theorem}

Since we aim to show that sharply 2-transitive groups of finite Morley rank are necessarily split, we  will be focusing on the failure of $J^2$ being a subgroup.

\section{A point-line geometry}\label{sec:geometry}

Let $G$ be a  sharply 2-transitive group of characteristic $\chara(G) \neq 2$ and let $J$ be the set of involutions. If commuting is transitive on the set of non-trivial translations in $J^2$, there is a well-behaved point-line geometry  first defined by Schr\"oder \cite{schroeder} following ideas of  Bachmann \cite{bachmann}. We follow the construction in Section 11 of \cite{borovik-nesin} although we explicitely define lines as subsets of $J$. 

\begin{lemma}\label{lem:geometry} If $\chara(G)\neq 2$, the following conditions are  equivalent:
\begin{enumerate}
\item  [(a)] Commuting is transitive on $J^2 \setminus \{1\}$.
\item  [(b)]  $iJ \cap kJ$ is uniquely 2-divisible for all involutions $i\neq k\in J$.
\item  [(c)]   $\Cen(ik) = iJ \cap kJ$ is abelian and is inverted by $k$ for all $i\neq k\in J$. 
\item  [(d)] The set $\{\Cen(\sigma) \setminus \{1\} : \sigma \in J^2 \setminus \{ 1 \} \}$ forms
a partition of $J^2 \setminus \{1\}$.
\end{enumerate}  
 
\end{lemma}
\begin{proof}
(a) $\Rightarrow$ (b):   Note that  since $(ij)^2=ii^j\in iJ$ every element of $iJ$ has a unique square-root in $iJ$.  Let $\tau \in iJ \cap kJ$. By assumption the group $A = \langle \Cen(\tau) \cap J^2 \rangle \leq \Cen(\tau)$ is abelian. Moreover, $A \cap J = \emptyset$ by Remark~\ref{lem:basic}. Hence the square-map is an injective group homomorphism from $A$ to $A$.
  
  There is $\sigma_i \in iJ$ such that $\sigma_i^2 = \tau$ and therefore $\sigma_i \in \Cen(\tau) \cap iJ$ because commuting is transitive. Similarly we  find $\sigma_k \in \Cen(\tau) \cap kJ$ such that $\sigma_k^2 = \tau$. Since the square-map is injective, it follows that $\sigma_i=\sigma_k \in iJ \cap kJ$. Therefore $iJ \cap kJ$ is uniquely 2-divisible. 

(b) $\Rightarrow$ (c) is contained in  Lemma 11.50 iv of \cite{borovik-nesin}.

(c) $\Rightarrow$ (d)  and  (d) $\Rightarrow$ (a) are obvious.
\end{proof}

Clearly, these conditions are satisfied in split sharply 2-transitive groups by Theorem~\ref{thm:neumann}.  Furthermore, by Lemma 11.50 of \cite{borovik-nesin}, these conditions are automatically satisfied whenever $\chara(G)=p\neq 0, 2$ or in case $G$ satisfies the descending chain condition for centralizers and hence in particular if  $G$ has finite Morley rank. On the other hand, the examples constructed in \cite{rips-segev-tent} (see also \cite{TentZiegler}) show that in characteristic 2 these conditions need not be satisfied. The non-split examples in characteristic 0 constructed in \cite{rips-tent} satisfy the assumptions and it is an open question whether non-split sharply 2-transitive groups exist in characteristic 0 which fail to satisfy these conditions.

If any of the conditions of Lemma~\ref{lem:geometry} is satisfied,  we obtain a point-line geometry as follows: the points of this geometry are the involutions of $G$. Given two points $i \neq j\in J$, we set
\[ \ell_{ij} = \{ k \in J : ij \in kJ \}\]
to be the (unique) line containing $i$ and $j$. 
By Lemma~\ref{lem:geometry}  we then have
 \[ \ell_{ij} = \{ k \in J : ij \in kJ \} = i\ \Cen(ij)=\{k\in J: (ij)^k=ji\}. \]

This implies that the point-line geometry is a partial projective plane: more precisely,  any two points are contained in a unique line and (hence) any two lines intersect in at most one point.

\begin{lemma} \label{line_lemma} Assume that  $G$ is sharply 2-transitive, $\chara(G)\neq 2$ and assume that commuting is transitive on $J^2\setminus\{1\}$.
	Let $\lambda$ be a line.
	\begin{enumerate}
		\item[(a)] Suppose $\lambda^i = \lambda^j$ for involutions $i \neq j$. Then $i,j \in \lambda$.
		\item[(b)] Suppose $\lambda \cap \lambda^i \neq \emptyset$ for some involution $i$. Then $i \in \lambda$. 
	\end{enumerate}
\end{lemma}
\begin{proof}
Part (a) is contained in the proof of Theorem 11.71 in \cite{borovik-nesin}, part (b) is Lemma 11.59 in \cite{borovik-nesin}.
Since our definition of lines is slightly different from the one given in \cite{borovik-nesin}, we include proofs.

	(a) If $\lambda^i = \lambda^j$ then $ij \in N_G(\lambda)$ and hence $ij \in N_G(\lambda^2)$. Now $\lambda^2 = \Cen(\sigma)$ for some $\sigma \in J^2 \setminus \{ 1 \}$ such that $\lambda = 
	\ell_\sigma$. Fix $s \in \lambda$. The group $N_G(\Cen(\sigma)) = \Cen(\sigma) \rtimes N_{\Cen(s)}(\Cen(\sigma))$ is split sharply 2-transitive by Proposition 11.51 of \cite{borovik-nesin}.
	Hence
	\[ ij \in N_G(\Cen(\sigma)) \cap J^2 = \Cen(\sigma) \]
	and therefore $i,j \in \ell_\sigma = \lambda$.
	
	(b) We may assume $\lambda \neq \lambda^i$. Hence there must be a unique $j \in \lambda \cap \lambda^i$. But then $j$ is fixed by $i$ and by Lemma~\ref{lem:basic} (b) we have $i = j \in \lambda$.
\end{proof}

We first observe that the geometry associated to such a group $G$ does not contain a proper projective plane.

\begin{lemma} Let $G$ be as in Lemma~\ref{line_lemma} and 
let $H \subseteq J^2$ be a subgroup of $G$ which is uniquely 2-divisible and normalized by an involution $i \in J$. Then $H = \Cen(\sigma)$ for some $\sigma \in J^2 \setminus \{ 1 \}$.
\end{lemma}
\begin{proof}
	Since $H$ is uniquely 2-divisible and $i$ acts as an involutionary automorphism without fixed points, it follows as in \cite{neumann} that $H$ is abelian and hence must be contained in the centralizer of some translation.
\end{proof}

\begin{proposition}\label{prop:no proj plane}
Let $G$ be as in Lemma~\ref{line_lemma}.
  There is no proper projective plane $X \subseteq J$. I.e. if $X \subseteq J$ satisfies
  \begin{enumerate}
    \item[(a)] $\forall i \neq j \in X: \ell_{ij} \subseteq X$, and
    \item[(b)] if $\lambda$ and $\delta$ are lines contained in $X$ then $\lambda \cap \delta \neq \emptyset$,
  \end{enumerate}
  then $X$ contains at most one line.
\end{proposition}
\begin{proof}
  Suppose $X \subseteq J$ satisfies (a) and (b). Take $\sigma, \tau \in X^2 \setminus \{ 1 \}$ and let $i$ be a point in $\ell_\sigma \cap \ell_\tau$. We may write $\sigma = ai, \tau = ib$ for some $a \in \ell_\sigma, b \in \ell_\tau$. Then $\sigma\tau = ab \in X^2$. Therefore $X^2$ is closed under multiplication and thus must be a subgroup of $G$. Moreover, $X^2$ is uniquely 2-divisible since it is a union of centralizers of translations.
  
  Each $j \in X$ acts on $X^2$ as an involutionary automorphism without fixed points. By the previous lemma $X^2 \leq \Cen(\sigma)$ and hence $X \subseteq \ell_\sigma$.
\end{proof}

\section{Sharply 2-transitive groups of finite Morley rank}\label{sec:finiteMR}

Let $G$ be a sharply $2$-transitive group of finite Morley rank with $\chara(G) \neq 2$ and let $J$ denote the set of involutions in $G$.  By Lemma 11.50 of \cite{borovik-nesin}, $iJ \cap jJ$ is uniquely 2-divisible for all $i\neq j\in J$ and so we can use the point-line geometry introduced in the previous section.  
We set $n = \MR(J)$ and $k = \MR(\Cen(ij))$ for involutions $i \neq j$. Note that $k$ does not depend on the choice of $i$ and $j$.

Since $G$ acts sharply $2$-transitively on $J$, it is easy to see that $\MR(G) = 2n$ and $\MR(J^2) = 2n-k$. Moreover, $G$ and $\Cen(ij)$ have Morley degree 1 by Lemma 11.60 of~\cite{borovik-nesin}.  


\begin{proposition}\label{prop:subgroup generated by J^2} 
	\begin{enumerate}
		\item[(a)] The set $iJ$ is indecomposable for all $i \in J$.
		\item[(b)] $\langle J^2 \rangle$ is a definable connected subgroup. In particular, there is a bound $m$ such that any $g\in \langle J^2 \rangle$ is a product of at most $m$ translations.
		\item[(c)]   $J^2$ is not generic in $\langle J^2 \rangle$.
		\item[(d)]   $\MR(J^3) > \MR(J^2)$.
	\end{enumerate}
\end{proposition}
\begin{proof}
	(a) Fix an involution $i \in J$. The set $iJ$ is normalized by $\Cen(i)$, hence it suffices to check indecomposability for $\Cen(i)$-normal subgroups. If $H \leq G$ is a $\Cen(i)$-normal subgroup of $G$, then either $\Cen(ij) \leq H$ for all $j \in J \setminus \{i\}$ or $H\  \cap\ \Cen(ij)$ has infinite index in $\Cen(ij)$ for all $j \in J \setminus \{i\}$. Therefore the set $iJ = \bigcup_{j \in J \setminus \{i\}} \Cen(ij)$ is indecomposable.
	
	(b)  Since $\langle J^2 \rangle = \langle iJ \rangle$, this follows from  Zil'ber's indecomposability theorem using (a).
	
	(c) Fix two involutions $i \neq j$. We claim that  
	\[\MR( \{ \tau \in J^2 : i^\tau = j \}) \geq n-k.\]
	To see this note that for any $r\in J$ by Remark~\ref{lem:basic} there is a unique $s\in J$ such that $i^{rs}=j$. Hence  the set $T_{ij}=\{( r, s)\colon i^{rs}=j\}\subset J\times J$  has Morley rank $n$. The equivalence classes on $T_{ij}$ given by $(r,s)\equiv (r',s')$ if and only if $rs=r's'$ have Morley rank at most~$k$. Hence the claim follows.

In particular, for any $\sigma \in J^2 \setminus \{1\}$ and $i\in J$ the set
 $\Sigma_i=\{ \tau \in J^2 : i^\sigma = i^\tau\}$ has Morley rank at least $n-k$. Since for $i\neq j\in J$
 the sets $\Sigma_i$ and $\Sigma_j$ intersect only in $\sigma$, it follows that
 $\{ \tau \in J^2 : \exists i \in J: i^\sigma = i^\tau\}$ has Morley rank (at least) $2n-k = \MR(J^2)$. 
  Hence for every $\sigma \in J^2 \setminus \{1\}$ the set
	\[ \{ \tau \in J^2 : \sigma\tau^{-1} \text{ has a fixed point} \} \]
	is a generic subset of $J^2$. Since translations do not have fixed points, it follows that 
	\[ \MR(\sigma J^2 \cap J^2 ) < 2n-k.\]
	 Thus, $J^2 $  is not generic  in $ \langle J^2 \rangle$. 
	 
(d)  Suppose $\MR( J^2 ) = \MR( J^3 )$. Since $(iJ)^2=J^iJ=J^2$ and $(iJ)^3=iJ^3$ we have $\MR( (iJ)^2 )= \MR( (iJ)^3 )$ and by (the proof of) Zil'ber's indecomposability theorem we get $\MR( \langle iJ \rangle ) = \MR( (iJ)^2)$. In particular, $J^2 \subseteq \langle iJ \rangle$ is a generic subset contradicting~(c).
\end{proof}

\begin{remark}\label{rem:Rips-Tent}
By Proposition~\ref{prop:subgroup generated by J^2} (b) it is easy to see that the non-split examples of sharply 2-transitive groups of characteristic $0$ constructed in \cite{rips-tent} do not have finite Morley rank.
\end{remark}

\begin{corollary}
	If the lines are strongly minimal, then $G$ is simple.
\end{corollary}
\begin{proof}
	Let $N \neq 1$ be a normal subgroup of $G$. Fix an involution $i$ and an element $g \in N \setminus \Cen(i)$. Then $ii^g = (g^{-1})^ig \in N$ and therefore $N \cap J^2 \neq 1$. Since $J^2\setminus\{1\}$ is a conjugacy class, it follows that $J^2 \subseteq N$. If $k = 1$ then $N$ must be generic since $\MR(J^3) > \MR(J^2) = 2n-1$ and $iJ^3 \subseteq N$. Therefore $N = G$. 
\end{proof}

By Proposition 11.71 of \cite{borovik-nesin} we have the following inequality.

\begin{proposition}[Proposition 11.71 of \cite{borovik-nesin}]
	$0 < 2k < n$.
\end{proposition}

Its proof uses a line counting argument. We will need a slightly more general version.

\begin{lemma} \label{line_counting}
	Let $H \leq G$ be a definable subgroup such that $\MR( H \cap J) = 2k$ and $\Mdeg( H \cap J )=1$. Then $\MR(\{ \lambda : \lambda \text{ is a line s.t. } \lambda \subseteq H \cap J \} ) < 2k$.
\end{lemma}
\begin{proof}
	This is proved in the same way as Proposition 11.71 of \cite{borovik-nesin}.
	Put $Z = H \cap J$ and let $\Lambda$ be the set of lines contained in $Z$. Since each $\lambda \in \Lambda$ has Morley rank 2k many preimages in $Z \times Z$, we have $\MR(\Lambda) \leq 2k$.
	Now assume $\MR(\Lambda) = 2k$. By the above argument we have $\Mdeg(\Lambda) = 1$ since $\Mdeg(Z) = 1$.
	
	Let $\lambda \in \Lambda$ be a line. By Lemma \ref{line_lemma} the family $( \lambda^i : i \in Z \setminus \lambda )$ consists of Morley rank 2k many lines which do not intersect $\lambda$. Hence the set $\{ \delta \in \Lambda : \lambda \cap \delta = \emptyset \} \subseteq \Lambda$ is a generic subset of $\Lambda$.
	
	We aim to find a line which intersects Morley rank 2k many lines contradicting $\Mdeg(\Lambda) = 1$. For $x \in Z$ set $\Lambda_x = \{ \lambda \in \Lambda : x \in \lambda \}$ and set $B(x) = \bigcup \Lambda_x \subseteq Z$. Note that
	$\MR(B(x)) = \MR(\Lambda_x)+ k $
	 and hence $\MR( \Lambda_x ) \leq k$ for all $x \in Z$.
	 Since each $\lambda \in \Lambda$ contains Morley rank $k$ many points, we must have $\MR(\Lambda_x) = k$ for a generic set of $x \in Z$.
	 
	 Fix $x_0 \in Z$ such that $\Lambda_{x_0}$ has Morley rank 2k. Then $B(x_0) \subseteq Z$ is generic and hence $\MR(\Lambda_x) = k$ for a generic set of $x \in B(x_0)$. Since $B(x_0) = \bigcup \Lambda_{x_0}$, we can find a line $\lambda \in \Lambda_{x_0}$ such that $\MR(\Lambda_x) = k$ for a generic set of $x \in \lambda$. But then $\lambda$ intersects Morley rank 2k many lines in $\Lambda$.
\end{proof}

To improve this rank inequality, we need to consider generalizations of projective planes. For definable sets $X$ and $Y$, we write $X \approx Y$ if and only if $\MR(X \Delta Y) < \MR(X)$.

\begin{definition}\label{def:generic proj plane}
	A definable subset $X \subseteq J$ is a \emph{generic projective plane} if
	\begin{enumerate}
		\item[(a)] $\MR(X) = 2k$ and $\Mdeg(X) = 1$, and
		\item[(b)] $\MR(\Lambda_X) = 2k$ and $\Mdeg(\Lambda) = 1$,
	\end{enumerate}
	where $\Lambda_X$ is the set of all lines $\lambda \subseteq J$ such that $\lambda \cap X \approx \lambda$.
\end{definition}

The next lemma follows from easy counting arguments.
\begin{lemma}
  Let $X \subseteq J$ be a definable set of Morley rank 2k and Morley degree 1. The following are equivalent:
  \begin{itemize}
    \item [(a)] $X$ is a generic projective plane,
    \item [(b)] $\MR(\Lambda_X) \geq 2k$,
    \item [(c)] $\MR( \{ \lambda \in \Lambda_X : x \in \lambda \}) = k$ for a generic set of $x \in X$.
  \end{itemize}
\end{lemma}

\begin{lemma} Assume $X \subseteq J$ is a generic projective plane and let $Z \subseteq J$ be a definable subset such that $X \approx Z$. Then $Z$ is a generic projective plane.
\end{lemma}
\begin{proof}
	For $x \in X$ put $\Lambda_x = \{ \lambda \in \Lambda_X : x \in \lambda \}$. If $\MR( \Lambda_x ) = k$, then $B(x) = \bigcup \Lambda_x \approx X$. In particular, $B(x) \approx Z$ for a generic set of $x \in X \cap Z$. If $B(x) \approx Z$, then $\Lambda_x \cap \Lambda_Z$ must have Morley rank $k$. Hence it follows from the previous lemma, that $Z$ must be a generic projective plane.
\end{proof}

\begin{proposition} \label{no-plane}
	$G$ does not contain a generic projective plane $X \subseteq J$.
\end{proposition}
\begin{proof}
	Assume $X \subseteq J$ is a generic projective plane and
	put \[H = N_G^\approx(X) = \{ g \in G: X^g \approx X \}.\] By Lemma 4.3 of \cite{wagner} we can find $Z \subseteq J, Z \approx X$ such that $H \leq N_G(Z)$. 
	If $\Lambda_x = \{\lambda \in \Lambda_X : x \in \lambda \}$ has Morley rank k, then $\bigcup_{\lambda \in \Lambda_x}\lambda \approx X$ and hence $x \in H$.
	Thus $X \cap H \subseteq X$ is generic and therefore we may assume $Z \subseteq H$. Since $H$ normalizes $Z$, it follows from Lemma \ref{lem:basic} (b) that $Z$ must be generic in $H \cap J$. Hence we may assume $Z = H \cap J$.
	
	Note that if $\lambda$ is a line such that $\lambda \cap H \subseteq \lambda$ is generic, then $\lambda \subseteq H$. Each line contained in $H$ has rank 2k many preimages in $Z \times Z$. Since $X$ is a generic projective plane and $Z \approx X$, the previous lemma implies that $Z$ is a generic projective plane and hence the set of all lines in $H$ has rank 2k and degree 1.	
	This contradicts Lemma \ref{line_counting}.
\end{proof}

\begin{theorem}\label{thm:main}
	Set $l = \MR(J^3) - \MR(J^2) \geq 1$. Then $n > 2k+l$.
\end{theorem}
\begin{proof}
	Consider the multiplication map $\mu: J \times J \times J \rightarrow J^3$. For $\alpha \in J^3$ we set $X_\alpha$ to be the set
	\[ X_\alpha = \{ i \in J : \exists r,s \in J \; irs = \alpha \}. \]
	Equivalently, $X_\alpha = \{i \in J :  i\alpha \in J^2 \}$ is the set of all involutions $i$ such that $i\alpha$ is a translation.
	
	Since $\MR(J^3) = 2n-k+l$ there must be some $\alpha \in J^3 \setminus J$ such that $\mu^{-1}(\alpha) \leq n+k-l$. Set $X = X_\alpha$ for such an $\alpha \in J^3 \setminus J$. If $irs = \alpha$, then $\MR(\{ j \in J: rs \in jJ\}) = k$ and hence $\MR(\mu^{-1}(\alpha))= \MR(X) + k$. Therefore we have $\MR(X) \leq n-l$.
	
	We now aim to show that $2k < \MR(X)$. If $irs = \alpha$ and $v \in \ell_{rs}$, then $\ell_{iv} \subseteq X$: We have $irs = ivu$ for some $u \in \ell_{rs}$ and moreover for each $p \in \ell_{iv}$ there is some $q \in \ell_{iv}$ such that $pq = iv$ and hence $pqu = ivu = irs = \alpha$. Hence each point in $X$ is contained in Morley rank $k$ many lines which are contained in $X$. Hence $X$ must have Morley rank at least $2k$.
	
	Now assume $\MR(X) = 2k$ and set $m = \Mdeg(X)$. Let $\Lambda$ be the set of lines obtained as above. For each $x \in X$ the set $\{ \lambda \in \Lambda : x \in \lambda\}$ has Morley rank $k$ and Morley degree $1$.
	Write $X$ as the disjoint union of definable sets $X_1, \dots X_m$, each of Morley rank $2k$ and Morley degree $1$.
	For $a = 1, \dots m$ let $\Lambda_a \subseteq \Lambda$ be the set $\Lambda_a = \{ \lambda \in \Lambda : \lambda \cap X_a \approx \lambda \}$. Since each $x \in X$ is contained in Morley rank $k$ many lines in $\Lambda$ and each line in $\Lambda$ contains Morley rank $k$ many points, the set $\Lambda$ must have Morley rank $2k$. Hence there is $b$ such that $\Lambda_b$ has Morley rank $2k$. Now $X_b$ is a generic projective plane. This contradicts Proposition \ref{no-plane}.
	
	Therefore $2k < \MR(X) \leq n-l$ and hence $2k+l < n$.
\end{proof}

The previous theorem  implies:

\begin{corollary}
	If $G$ is a sharply 2-transitive group, \ $\MR(G) = 6$, then $G$ is of the form $\mathrm{AGL}_1(K)$ for some algebraically closed field $K$ of Morley rank $3$.
\end{corollary}
\begin{proof}
If $\chara(G)\neq 2$, this follows from Theorem~\ref{thm:main}. If $\chara(G)=2$, then $G$ is split by~\cite{altinel-berkman-wagner} and any point stabilizer has Morley rank 3. Since the point stabilizers do not contain involutions, they are solvable by~\cite{frecon}. Now the result follows from~\cite{borovik-nesin}, Cor.~11.66. 
\end{proof}

\bibliographystyle{plain}
\bibliography{bibliography}

\end{document}